\documentclass[11pt,letterpaper]{amsart}
\usepackage[utf8]{inputenc}
\usepackage{amsmath}
\usepackage{amsfonts}
\usepackage{amssymb}
\usepackage{graphicx}
\usepackage[left=2.4cm,right=2.4cm,top=2.5cm,bottom=2.5cm]{geometry}
\numberwithin{equation}{section}
\graphicspath{ {images/} }
\usepackage{amsmath,amssymb, amsthm} 
\usepackage{lipsum}
\usepackage{stmaryrd}
\usepackage{dsfont}

\makeatletter
\def\@settitle{\begin{center}%
  \baselineskip14\p@\relax
  \bfseries
  \uppercasenonmath\@title
  \@title
  \ifx\@subtitle\@empty\else
     \\[1ex]\uppercasenonmath\@subtitle
     \footnotesize\mdseries\@subtitle
  \fi
  \end{center}%
}
\def\subtitle#1{\gdef\@subtitle{#1}}
\def\@subtitle{}
\makeatother

\usepackage{tikz-cd}
\usepackage{hyperref}
\usepackage{bm}
\hypersetup{colorlinks, linkcolor=blue, citecolor=black, urlcolor=black}

\usepackage[english]{babel}
\usepackage[colorinlistoftodos]{todonotes}

\theoremstyle{plain}
\newtheorem{thm}{Theorem}[subsection] % reset theorem numbering for each chapter

\usepackage{mathrsfs}

\theoremstyle{definition}
\newtheorem{defi}[thm]{Definition}

\newtheorem{rem}[thm]{Remark}
\theoremstyle{definition}

\theoremstyle{plain}
\newtheorem{prop}[thm]{Proposition}
\theoremstyle{plain}
\newtheorem{lemma}[thm]{Lemma}
\theoremstyle{plain}

\theoremstyle{plain}

\newtheorem{thmintro}{Theorem}


\newcounter{parentnumber}

\newcommand{\bQ}{\mathbb{Q}}

\makeatletter  
\newcommand{\colim@}[2]{%
  \vtop{\m@th\ialign{##\cr
    \hfil$#1\operator@font colim$\hfil\cr
    \noalign{\nointerlineskip\kern1.5\ex@}#2\cr
    \noalign{\nointerlineskip\kern-\ex@}\cr}}%
}
\newcommand{\colim}{%
  \mathop{\mathpalette\colim@{\rightarrowfill@\scriptscriptstyle}}\nmlimits@
}
\renewcommand{\varprojlim}{%
  \mathop{\mathpalette\varlim@{\leftarrowfill@\scriptscriptstyle}}\nmlimits@
}
\renewcommand{\varinjlim}{%
  \mathop{\mathpalette\varlim@{\rightarrowfill@\scriptscriptstyle}}\nmlimits@
}
\makeatother

\newcommand{\Z}{\mathbb{Z}}
\newcommand{\Q}{\mathbb{Q}}

\newcommand{\X}{\mathcal{X}}

\font\wncyr=wncyr9.8
\newcommand{\sha}{\text{\wncyr{W}}}

\newcommand{\rH}{\mathrm{H}}

\begin{document}
\title{A formula of Perrin-Riou and characteristic power series of signed Selmer groups}

\author{Francesc Castella}
\address{University of California Santa Barbara, South Hall, Santa Barbara, CA 93106, USA}
\email{castella@ucsb.edu}

\date{\today}

\dedicatory{In memory of Professor John Coates}
\thanks{This research was partially supported by the NSF grants DMS-2101458 and DMS-2401321.}

\begin{abstract}
We prove a conjecture of Kundu--Ray, following from the $p$-adic Birch--Swinnerton-Dyer conjecture for supersingular primes by Bernardi--Perrin-Riou and Kato's Main Conjecture, predicting an expression  for the leading term (up to a $p$-adic unit) of a characteristic power series of Kobayashi's signed Selmer groups attached to elliptic curves $E/\Q$ with supersingular reduction at a prime $p>2$ with $a_p=0$. The proof is deduced from a similar formula due to Perrin-Riou for a generator of her module of arithmetic $p$-adic $L$-functions with values in the Dieudonn\'{e} module of $E$. 
\end{abstract}

\maketitle
\tableofcontents

%\section*{Introduction}
\addtocontents{toc}{\protect\setcounter{tocdepth}{2}}
%\addtocontents{lof}{\protect\setcounter{tocdepth}{0}}
%\renewcommand{\thetheorem}{\Alph{thm}}

\section{Introduction}

Let $E/\Q$ be an elliptic curve and $p$ an odd prime of good reduction for $E$. Let $\X(E/\Q_\infty)$ denote the Pontryagin dual of the Selmer group ${\rm Sel}_{p^\infty}(E/\Q_\infty)$  over the cyclotomic $\Z_p$-extension of $\Q_\infty/\Q$. Let $\Lambda=\Z_p[[{\rm Gal}(\Q_\infty/\Q)]]$ be the cyclotomic Iwasawa algebra, which we identify with the one-variable power series ring $\Z_p[[X]]$ upon the choice of a topological generator $\gamma\in{\rm Gal}(\Q_\infty/\Q)$.

 When $p$ is ordinary for $E$, the Selmer group $\X(E/\Q_\infty)$ is known to be $\Lambda$-torsion by work of Kato \cite{kato-295}, and letting $\xi_p\in\Lambda=\Z_p[[X]]$ denote a characteristic power series for $\mathcal{X}(E/\Q_\infty)$, the work of Schneider \cite{schneider-II} and Perrin-Riou \cite{PR-abvar} (see also \cite{PR-MSMF} for the case where $E$ has complex multiplication) yields an  analogue of the Birch--Swinnerton-Dyer conjecture for $\xi_p$, relating its order of vanishing at $X=0$ to the Mordell--Weil rank of $E$, and expressing its leading coefficient in terms of arithmetic invariants of $E$. 
 
The goal of this note is to prove an analogous result in the case where $p$ is a prime of supersingular reduction for $E$ with $a_p=0$. 
%(a condition that holds automatically for supersingular $p>3$). 
Our main result is in terms of a characteristic power series of Kobayashi's signed Selmer groups; in the rank zero case, 
% it recovers a result originally due to 
a result along these lines was first proved by B.-D.~Kim \cite{BDKim-euler-char} by an adaptation of Greenberg's methods \cite{greenberg-cetraro}, so we focus on the case of Mordell--Weil rank $r\geq 1$, where the result we obtain was
% (in the special case $F=\Q$), and in positive rank it proves a folklore conjecture stated 
conjectured by Kundu--Ray (\emph{cf.} \cite[Conjecture~3.15]{KR-I}), following Sprung's reformulation \cite{sprung-RNT} of the $p$-adic Birch--Swinnerton-Dyer conjecture of Bernardi and Perrin-Riou \cite{BPR} (see also Remark~\ref{rem:pBSD-conj} below). 
 
\subsection{Main result}\label{subsec:main}

From now on, assume that $p>2$ is a supersingular prime for $E$ satisfying $a_p=0$ (a condition that holds automatically unless $p=3$). In \cite{kobayashi}, Kobayashi introduced signed Selmer groups ${\rm Sel}_{p^\infty}^\pm(E/\Q_\infty)$  whose Pontryagin dual 
\[
\X^\pm(E/\Q_\infty)={\rm Hom}_{\Z_p}({\rm Sel}_{p^\infty}^\pm(E/\Q_\infty),\Q_p/\Z_p)
\] 
he showed to be $\Lambda$-torsion %(is contrast to their classical counterparts) 
as a consequence of Kato's work. 

As explained in the work of Bernardi--Perrin-Riou \cite{BPR}, one can naturally attach a quadratic form $h_\nu$ on $E(\Q)$ to every vector $\nu$ in the Dieudonn\'{e} module $D_p(E)=\Q_p\otimes_{\Q}\rH^1_{\rm dR}(E/\Q)$, and we let ${\rm Reg}_\nu\in\Q_p$ be the discriminant of the associated bilinear ($p$-adic height) pairing  
\[
\langle\cdot,\cdot\rangle_\nu:E(\Q)\times E(\Q)\rightarrow\Q_p.
\]
By linearity, these can be extended to $E(\Q)\otimes\Z_p$. Consider the \emph{strict} (or \emph{fine}, in the terminology of e.g. \cite{wuthrich-JAG}) Mordell--Weil group
\[
(E(\Q)\otimes_{}\Z_p)_0:={\rm ker}\bigl\{E(\Q)\otimes_{}\Z_p\overset{}\rightarrow E(\Q_p)\hat\otimes\Z_p\bigr\},
\]
where  $E(\Q_p)\hat\otimes\Z_p$ is the $p$-adic completion of $E(\Q_p)$.

In Section~\ref{sec:coordinates}, similarly as in the work of Sprung \cite{sprung-RNT} we shall introduce certain vectors $N^\pm\in D_p(E)$ in the complement to the Hodge filtration ${\rm Fil}^0D_p(E)=\Q_p\omega_E$, where $\omega_E$ is a N\'{e}ron differential on $E$. Write ${\rm Reg}_p^\pm$ (resp. ${\rm Reg}_p^{\rm str}$) for the above regulator on $E(\Q)$ (resp. $E(\Q)\otimes\Z_p)_0$) associated to 
\[
h_{N^\pm/[\omega_E,N^\pm]_{\rm dR}}=h_{N^\pm}/[\omega_E,N^\pm]_{\rm dR},
\]
where $[\cdot,\cdot]_{\rm dR}$ denotes the de~Rham pairing on $D_p(E)$.

Let $\kappa:{\rm Gal}(\Q_\infty/\Q)\simeq
%\simeq{\rm Gal}(\Q(\mu_{p^\infty})/\Q(\mu_p))\xrightarrow{\simeq} 
1+p\Z_p$ be the isomorphism defined by the $p$-adic cyclotomic character. The main result of this note is the following $p$-adic analogue of the Birch--Swinnerton-Dyer conjecture for supersingular primes.
% for the characteristic power series of signed Selmer groups.

\begin{thmintro}\label{thm:main-intro}
Let $E/\Q$ be an elliptic curve with good supersingular reduction at an odd prime $p$ with $a_p=0$. Put 
\[
r={\rm rank}_\Z E(\Q) 
\]
and suppose $r\geq 1$. Let $\xi_p^\pm\in\Lambda\simeq\Z_p[[X]]$ be a characteristic power 
series for $\X^\pm(E/\Q_\infty)$. 
%and put 
%\[
%\varrho:=\min\{{\rm ord}_X(\xi^+),{\rm ord}_X(\xi^-)\}
%\] 
Then:
\begin{itemize}
\item[(i)] %$\xi_p^\pm$ vanishes to order at least $r$ at $X=0$, i.e. $\varrho\geq r$.
$\varrho:=\min\{{\rm ord}_X(\xi^+_p),{\rm ord}_X(\xi^-_p)\}\geq r$.
\item[(ii)]  If $\sha(E/\Q)[p^\infty]$ is finite and ${\rm Reg}_p^{\rm str}\neq 0$, then equality holds in ${\rm (i)}$, and the leading coefficient $(\xi_p^{+,\divideontimes},\xi_p^{-,\divideontimes})$ of the vector $(\xi_p^+,\xi_p^-)\in\Z_p[[X]]^{\oplus{2}}$ is given up to a $p$-adic unit by  
\[
(\xi_p^{+,\divideontimes},\xi_p^{-,\divideontimes})\;\sim_p\;(\log_p\kappa(\gamma))^{-r}\cdot({\rm Reg}_p^+,{\rm Reg}_p^-)\cdot\frac{\#\sha(E/\mathbb{Q})[p^\infty]\cdot{\rm Tam}(E/\mathbb{Q})}{(\#E(\mathbb{Q})_{\rm tors})^2},
\]
where ${\rm log}_p$ is Iwasawa's branch of the $p$-adic logarithm and ${\rm Tam}(E/\Q)=\prod_\ell c_\ell$ is the product of the local Tamagawa numbers of $E$.
\end{itemize}
%Hence \cite[Conjecture~3.15]{KR-I} holds true.
\end{thmintro}

%\begin{rem}
%For $r=0$, Theorem~\ref{thm:main-intro} gives a new proof of \cite[Thm.~1.2]{BDKim-euler-char} for $F=\Q$, while for $r>0$, the result is new and 
%We also note that this result is predicted by the combination of the $p$-adic Birch--Swinnerton-Dyer conjecture of Bernardi--Perrin-Riou \cite{BPR} (see also \cite[Conj.~2.5]{PR-expmath}) and Kato's Main Conjecture (in the formulation of \cite[\S{3.4}]{PR-points}. We note however, that the proof of Theorem~\ref{thm:main-intro} \emph{does not assume} the Main Conjecture.
%\end{rem}

\begin{rem}\label{rem:pBSD-conj}
The conclusion of Theorem~\ref{thm:main-intro} is predicted by  the combination of:
\begin{itemize}
\item The $p$-adic Birch--Swinnerton-Dyer conjecture for supersingular primes $p$ formulated by Bernardi--Perrin-Riou \cite{BPR} (see also \cite[Conj.~2.5]{PR-expmath}), as reformulated by Sprung \cite{sprung-RNT} in terms of signed $p$-adic $L$-functions $L_p^\pm$;
\item Kato's Main Conjecture (see \cite[\S{3.4}]{PR-points}), 
which is known to be equivalent to Kobayashi's Main Conjecture predicting the equality
\[
\bigl(\xi_p^\pm\bigl)\overset{?}=\bigl(L_p^\pm\bigr)
\]
as principal ideals in $\Lambda$ (see \cite[Thm.~7.4]{kobayashi}). 
\end{itemize}
%(We note however, that the proof of Theorem~\ref{thm:main-intro} \emph{does not assume} the Main Conjecture.) 
This prediction is recorded in \cite[Conjecture~3.15]{KR-I} (which the added prediction that ${\rm ord}_X(\xi_p^\pm)$ is independent of the choice of sign $\pm$).
%(and so, it provide some further evidence for it).
\end{rem}

\begin{rem}\label{rem:euler-char}
In \cite{ray-sujatha-CJM}, for any two elliptic curves $E_1,E_2$ over $\Q$ with good supersingular reduction at a prime $p>2$ with $a_p=0$ and $E_1[p]\simeq E_2[p]$ as $G_\Q$-modules, Ray and Sujatha establish relations mod $p$ between the ($\Sigma$-imprimitive) \emph{truncated $\Gamma$-Euler characteristics} of their respective signed Selmer groups ${\rm Sel}_{p^\infty}^\pm(E_i/\Q_\infty)$, defined as 
\[
\chi_t^\pm(\Gamma,E_i):=\frac{\#{\rm ker}(\phi_{E_i})}{\#{\rm coker}(\phi_{E_i})},
\]
where $\phi_{E_i}:{\rm Sel}_{p^\infty}^\pm(E_i/\Q_\infty)^\Gamma\rightarrow{\rm Sel}_{p^\infty}^\pm(E_i/\Q_\infty)_\Gamma$, sending $s\mapsto s\,{\rm mod}\,(\gamma-1)$, is the natural map  from the $\Gamma$-invariants to the $\Gamma$-coinvariants.
%betwen the $\Gamma$-invariants and the $\Gamma$-coinvariants of ${\rm Sel}_{p^\infty}^\pm(E/\Q_\infty)$. 
Using the structure theorem for finitely generated $\Lambda$-modules, one easily checks that $\chi_t^\pm(\Gamma,E_i)$ is defined,  in the sense that both $\ker(\phi_{E_i})$ and ${\rm coker}(\phi_{E_i})$ are finite, whenever $\gamma$ acts semi-simply on ${\rm Sel}_{p^\infty}^\pm(E_i/\Q_\infty)$ (a condition expected to always hold in the cyclotomic setting; see \cite[Lem.~6.1]{APAW} for a more general result), in which case one has
\[
\chi_t^\pm(\Gamma,E_i)=\vert\xi_p^{\pm,\divideontimes}(E_i)\vert_p^{-1}
\] 
(see e.g. \cite[Lem.~2.11]{zerbes-PLMS}), where $\xi_p^{\pm,\divideontimes}(E_i)$ is the leading coefficient of a characteristic power series %$\xi_p^{\pm,\divideontimes}(E_i)\in\Z_p[[X]]$ 
for $\mathcal{X}^\pm(E_i/\Q_\infty)$, and $\vert\cdot\vert_p$ denotes the $p$-adic absolute value on $\Q_p$ with $\vert p\vert_p=1/p$. Thus in particular \cite[Thm.~5.5]{ray-sujatha-CJM} (for at least one of the signs $\pm$)  becomes a congruence relation mod $p$ between the arithmetic invariants appearing in Theorem~\ref{thm:main-intro}.  

In a similar vein, as a consequence of Theorem~\ref{thm:main-intro}, \cite[Thm.~3.16]{KR-I} (for at least one of the signs $\pm$) now holds unconditionally. 
\end{rem}

%\begin{thmintro}\label{cor:B}
%Assume Kato's main conjecture. Then $L_p^\pm$ satisfies the $p$-adic Birch--Swinnerton-Dyer conjecture of \cite[Conj.~1.2]{sprung-RNT} up to a $p$-adic unit.
%\end{thmintro}

%Finally, we note that a proof of Kato's main conjecuture (in its equivalent form given by Kobayashi) has recently been announced in \cite{BSTW} for semistable elliptic curves.

\subsection{Outline of the proof}

In \cite{PR-points}, Perrin-Riou proved a $p$-adic Birch--Swinnerton-Dyer formula for a certain arithmetic $p$-adic $L$-function
\[
\mathcal{F}_p^{\rm PR}\in D_p(E)\otimes_{\Q_p}\mathcal{H},
\]
where %$D_p(E)=\Q_p\otimes_{\Q}{\rm H}^1_{\rm dR}(E/\Q)$ is the Dieudonn\'{e} module attached to $E$ and 
$\mathcal{H}\subset\Q_p[[X]]$ is the ring of power series convergent in the $p$-adic open unit disk. A term in her leading coefficient formula is a $p$-adic regulator
\begin{equation}\label{eq:intro-reg}
(1-\varphi)^2{\rm Reg}_p^{\rm PR}\in D_p(E)
\end{equation}
%attached to a $D_p(E)$-valued height pairing on $E(\Q)$
where $\varphi$ is the Frobenius operator. Building on a result of Lei \cite{lei-PhD} expressing Kobayashi's signed Coleman maps in terms of Perrin-Riou's work \cite{PR-94}, 
we extract from $\mathcal{F}_p^{\rm PR}$ two signed power series $\mathcal{F}_p^\pm\in\Z_p[[X]]$. 
%We note that the perspective introduced in the latter work, interpreting Kobayashi's construction of signed Coleman power series maps from the lens of Perrin-Riou's work \cite{PR-94} is a key input in this paper, and it might be useful in future generalizations.  
By direct computation of the coordinates of \eqref{eq:intro-reg} relative to a certain basis $(\nu_-,\nu_+)$ of $D_p(E)$ on the one hand, and of the same coordinates of the leading coefficient $\mathcal{F}_p^{{\rm PR},\divideontimes}\in D_p(E)$ of $\mathcal{F}_p^{\rm PR}$ on the other hand,   
%-- very similar to those perfomed by Sprung \cite{sprung-RNT} in his study of $p$-adic analogues of the Birch--Swinnerton-Dyer conjecture for the signed $p$-adic $L$-functions of \cite{sprung-JNT} --, 
from Perrin-Riou's formula we arrive at expressions for the order of vanishing and the leading coefficient of $\mathcal{F}_p^\pm$ agreeing with those in Theorem~\ref{thm:main-intro} for the characteristic power series $\xi_p^\pm$. 
%by means of computations very similar to those performed by Sprung \cite{sprung-RNT} in his study of $p$-adic analogues of the Birch--Swinnerton-Dyer conjecture for the signed $p$-adic $L$-functions constructed in \cite{sprung-JNT}, 
%
We note here that similar computations (that in fact served as the original motivation for this note) were performed by Sprung \cite{sprung-RNT} in his study of $p$-adic analogues of the Birch--Swinnerton-Dyer conjecture for the signed (or rather, $\sharp/\flat$-) $p$-adic $L$-functions constructed in \cite{sprung-JNT}.  
Finally, from an application of global duality we show that $\mathcal{F}_p^\pm$ generates the characteristic ideal of $\X^\pm(E/\Q_\infty)$. 
%thereby yielding Theorem~\ref{thm:main-intro}.

\subsection{Acknowledgements}  We would like to dedicate this note to the memory of Prof.~John Coates. It is a pleasure to thank Professors Ye Tian, Yichao Tian, and Xin Wan for the opportunity to make this small contribution to a special issue honoring such a great mathematician. We also thank Anwesh Ray and Debanjana Kundu for their comments on the topic of this note, and Antonio Lei for bringing \cite{ray-sujatha-CJM} to our attention.

\section{A formula of Perrin-Riou}

In this section we recall a $p$-adic Birch--Swinnerton-Dyer  formula for arithmetic $p$-adic $L$-functions established in \cite{PR-points}.

\subsection{Dieudonn\'{e} modules}

Let $E/\bQ$ be an elliptic curve, and $p$ an odd prime of good reduction for $E$. Let
\[
D_p(E):=\Q_p\otimes_{\Q}\rH^1_{\rm dR}(E/\Q)
\]
denote the Dieudonn\'{e} module of $E$. This is a $2$-dimensional $\Q_p$-vector space equipped with a Frobenius operator $\varphi$, a Hodge filtration $D_p(E)\supset{\rm Fil}^0D_p(E)\supset 0$, with ${\rm Fil}^0D_p(E)$ spanned by the class of a N\'{e}ron differential $\omega_E\in\Omega_{E/\Z}$, and a non-degenerate alternating pairing
\[
[\cdot,\cdot]_{\rm dR}:D_p(E)\times D_p(E)\rightarrow\Q_p.
\]
The operator $\varphi$ has characteristic polynomial $x^2-\frac{a_p}{p}x+\frac{1}{p}$, where $a_p:=p+1-\#E(\mathbb{F}_p)$.

\subsection{Arithmetic $p$-adic $L$-function}

Let $T$ be the $p$-adic Tate module of $E$, and put $V=\Q_p\otimes_{\Z_p}T$. As in the Introduction, let $\Gamma$ be the Galois group of the cyclotomic $\Z_p$-extension $\Q_\infty/\Q$, which we shall often identify with the Galois group of the cyclotomic $\Z_p$-extension $\Q_{p,\infty}/\Q_p$, and let $\Lambda=\Z_p[[\Gamma]]$ be the cyclotomic Iwasawa algebra, often identified with the $1$-variable power series ring $\Z_p[[X]]$ via $\gamma=1+X$ upon the choice of a fixed topological generator $\gamma\in\Gamma$. For each $n\geq 0$, let $\Q_n$ (resp. $\Q_{p,n}$) be the unique subextension of $\Q_\infty$ (resp. $\Q_{p,\infty}$) of degree $p^n$ over $\Q$ (resp. $\Q_p$).

For $h\geq 0$, let 
\[
\mathcal{H}_h=\Biggl\{\sum_{n\geq 0}c_nX^n\in\Q_p[[X]]\,\Bigl\vert\,\lim_{n\to\infty}\frac{\vert c_n\vert_p}{n^h}=0\Biggr\},
\]
where $\vert\cdot\vert_p$ denotes the $p$-adic absolute value on $\Q_p$ with the standard normalization $\vert p\vert_p=1/p$, and put $\mathcal{H}=\bigcup_{h\geq 0}\mathcal{H}_h$ and $\mathcal{H}(\Gamma)=\{f(\gamma-1)\,\vert\,f\in\mathcal{H}\}$. Write
\[
\rH^1_{\rm Iw}(\Q_{p,\infty},T):=\varprojlim_n\rH^1(\Q_{p,n},T)
\]
for the Iwasawa cohomology of $T$, and put $\rH^1_{\rm Iw}(\Q_{p,\infty},V)=\Q_p\otimes_{\Z_p}\rH^1_{\rm Iw}(\bQ_{p,\infty},T)$.

We begin by recalling Perrin-Riou's big exponential map, which we state below in a rather rough form (see e.g. \cite[\S{1}]{PR-points} for a more precise statement). The Weil pairing gives a natural identification $V\simeq V^*(1):={\rm Hom}_{\Q_p}(V,\Q_p(1))$ (so in particular, $\mathbf{D}_{\rm dR}(V^*(1)):=(V^*(1)\otimes_{\Q_p}\mathbf{B}_{\rm dR})^{G_{\Q_p}}\simeq D_p(E)$ by the comparison isomorphism), but in the following we shall nonetheless  keep the distinction between the two. 

\begin{thm}\label{thm:PR-map}
There exists an injective $\Lambda$-module homomorphism
\[
\Omega_{V^*(1)}:
\Lambda\otimes_{\Z_p}\mathbf{D}_{\rm dR}(V^*(1))\rightarrow\rH^1_{\rm Iw}(\Q_{p,\infty},V^*(1))\otimes_{\Q_p}\mathcal{H}(\Gamma)
\]
interpolating ${\rm exp}_{\bQ_{p,n},V^*(1)}$ for all $n\geq 0$.
\end{thm}

\begin{proof}
This follows by taking $h=1$ and $j=0$ in \cite[\S{3.2.3}]{PR-94} (see also \cite[Thm.~1.3]{PR-points}.
\end{proof}

For any $\eta\in\mathbf{D}_{\rm dR}(V^*(1))$, the map $\Omega_{V^*(1)}$ may be evaluated at $\eta\otimes(1+X)$.  Given $\mathbf{z}\in\rH^1_{\rm Iw}(\Q_{p,\infty},V)$, we thus define
\begin{align*}
\mathscr{L}_{\mathbf{z}}:\mathbf{D}_{\rm dR}(V^*(1))&\rightarrow\mathcal{H}(\Gamma),\quad
\eta\mapsto\left\langle\Omega_{V^*(1)}(\eta\otimes(1+X)),\mathbf{z}\right\rangle_{\Q_{p,\infty}},
\end{align*}
%by $\eta\mapsto\left\langle\Omega_{V^*(1)}(\eta\otimes(1+T)),\mathbf{z}\right\rangle_{\Q_{p,\infty}}$,
where $\langle\cdot,\cdot\rangle_{\Q_{p,\infty}}:\rH^1_{\rm Iw}(\Q_{p,\infty},V^*(1))\times\rH^1_{\rm Iw}(\Q_{p,\infty},V)\rightarrow\Lambda$ is Perrin-Riou's $\Lambda$-adic  Tate pairing (see e.g. \cite[\S{2.1.2}]{PR-points}), given by
\[
\langle\mathbf{x},\mathbf{y}\rangle_{\Q_{p,\infty}}:=\Biggl(\sum_{\sigma\in\Gamma_n}\langle x_n^{\sigma^{-1}},y_n\rangle_{\Q_{p,n}}\cdot\sigma\Biggr)_n
\]
for $\mathbf{x}=(x_n)_n$, $\mathbf{y}=(y_n)_n$ and $\Gamma_n={\rm Gal}(\bQ_{p,n}/\Q_p)$.

In the following, we shall view $\mathscr{L}_{\mathbf{z}}$ as an element
\[
\mathscr{L}_{\mathbf{z}}\in{D}_p(E)\otimes_{\Q_p}\mathcal{H}(\Gamma)
\]
using the canonical isomorphism ${\rm Hom}_{\Q_p}(\mathbf{D}_{\rm dR}(V^*(1)),\mathcal{H}(\Gamma))\simeq\mathbf{D}_{\rm dR}(V)\otimes_{\Q_p}\mathcal{H}(\Gamma)$ induced by $[\cdot,\cdot]_{\rm dR}$ and the identification $\mathbf{D}_{\rm dR}(V)\simeq D_p(E)$ arising from the comparison isomorphsm.

Let ${\rm Sel}_{p^\infty}^{\rm str}(E/\Q_n)$ be the \emph{strict Selmer group} defined by
\[
{\rm Sel}^{\rm str}_{p^\infty}(E/\Q_n):={\rm ker}\left\{{\rm Sel}_{p^\infty}(E/\Q_n)\xrightarrow{{\rm res}_p}E(\Q_{p,n})\otimes\Q_p/\Z_p\right\},
\]
and put ${\rm Sel}_{p^\infty}^{\rm str}(E/\Q_\infty)=\varinjlim_n{\rm Sel}_{p^\infty}^{\rm str}(E/\Q_n)$. For any finite set of primes $S$ containing $p$ and $\infty$, let $\Q^S$ denote the maximal extension of $\Q$ unramified outside $S$, and put
\[
\mathbb{H}^1(T):=\varprojlim_n\rH^1({\rm Gal}(\Q^S/\Q_n),T).
\]
%where the inverse limit is with respect to the corestriction maps. 
(This is easily checked to be independent of $S$; see e.g \cite[p.\,983]{PR-points}.) 

By Kato's work \cite{kato-295}, ${\rm Sel}_{p^\infty}^{\rm str}(E/\Q_\infty)$ is $\Lambda$-cotorsion and $\mathbb{H}^1(T)$ is torsion-free of $\Lambda$-rank 1.

\begin{defi}\label{def:arith-Lp}
Let $\mathbf{z}\in\mathbb{H}^1(T)$ be a nonzero element, and put
\[
\mathcal{F}_p^{\rm PR}:=\mathscr{L}_{\mathbf{z}_p}\cdot\frac{g_{\rm str}}{h_{\mathbf{z}}}\in D_p(E)[[X]],
\]
where $\mathbf{z}_p={\rm res}_p(\mathbf{z})$ denotes the image of $\mathbf{z}$ under the restriction map $\mathbb{H}^1(T)\rightarrow\rH^1_{\rm Iw}(\Q_{p,\infty},T)$ and $g_{\rm str}$ (resp. $h_{\mathbf{z}}$) is a characteristic power series for ${\rm Sel}_{p^\infty}^{\rm str}(E/\Q_\infty)^\vee$ (resp. $\mathbb{H}^1(T)/(\mathbf{z})$).
\end{defi}

We note that $\mathcal{F}_p^{\rm PR}$ gives a generator of the $\Lambda$-module of \emph{arithmetic $p$-adic $L$-functions} as introduced in Perrin-Riou's work (see e.g. \cite[\S{3.4.3}]{PR-points} and \cite[\S{3.1}]{PR-expmath}).

\subsection{$p$-adic regulators}\label{subsec:regulators}

%Recall that $\omega_E\in D_p(E)$ denotes the class of a N\'{e}ron differential of $E$ over $\Z$; 
Let 
\[
y^2-a_1xy+a_3y=x^3+a_2x^2+a_4x+a_6
\] 
be a minimal Weierstrass model for $E$. Take $\omega_E=\frac{dx}{2y+a_1x+a_3}$ and put $\eta=x\omega_E$; then the pair $(\omega_E,\eta)$ forms a basis for $D_p(E)$.

For each $\nu\in D_p(E)$, we let $h_\nu$ be the quadratic form on $E(\Q)$ defined as in \cite{BPR}. In particular, $h_\nu(P)=-\log_{\omega_E}(P)^2$, where $\log_{\omega_E}$ is the logarithm on $E(\Q)$ associated to $\omega_E$, $h_\eta$ is Bernardi's $p$-adic height using $p$-adic $\sigma$-functions \cite{bernardi}, and $h_{\nu}$ for an arbitrary $\nu=a\omega_E+b\eta\in D_p(E)$ is defined by linearity as $ah_{\omega_E}+bh_\eta$.

\begin{defi}\label{def:reg-nu}
Let $r={\rm rank}_{\Z}E(\Q)$, and let ${\rm Reg}_\nu$ denote the discriminant of the quadratic form $\langle P,Q\rangle_\nu:=h_\nu(P+Q)-h_\nu(P)-h_\nu(Q)$ on $E(\Q)$, i.e.
\begin{equation}\label{eq:reg}
{\rm Reg}_\nu=\frac{{\rm det}(\langle P_i,P_j\rangle_\nu)}{[E(\Q):\sum_{i=1}^r\Z P_i]^2},
\end{equation}
where $P_1,\dots,P_r$ is any system of $r$ points in $E(\Q)$ giving a basis of $E(\Q)\otimes_\Z\Q$.
\end{defi}

\begin{lemma}\label{lem:reg-p}
Suppose $r={\rm rank}_\Z E(\Q)\geq 1$. Then there exists a unique ${\rm Reg}_p^{\rm PR}\in D_p(E)$ such that
\[
\left[{\rm Reg}_p^{\rm PR},\nu\right]_{\rm dR}=\widetilde{\rm Reg}_\nu,\quad\textrm{where}\quad\widetilde{\rm Reg}_{\nu}:=\frac{{\rm Reg}_\nu}{[\omega_E,\nu]_{\rm dR}^{r-1}}
\]
for all %$\nu\in D_p(E)\smallsetminus\Q_p\omega_E$
$\nu\not\in{\rm Fil}^0D_p(E)$. 
\end{lemma}

\begin{proof}
This is shown in \cite[Lem.~2.6]{PR-expmath} (whose  statement is missing the factor $[\omega_E,\nu]^{r-1}$ as noted in \cite[Lem.~4.2]{SW}).
\end{proof}

%\begin{defi}
%For $r=0$, put ${\rm Reg}_p^{\rm PR}:=\omega_E$.
%\end{defi}

As in the Introduction, let $(E(\Q)\otimes\Z_p)_0\subset E(\Q)\otimes\Z_p$
%:={\rm ker}\bigl\{{\rm res}_p:E(\Q)\otimes\Z_p\rightarrow E(\Q_p)\hat\otimes\Z_p\bigr\}$ 
be the strict Mordell--Weil group. 

\begin{defi}\label{def:reg-str}
Write ${\rm Reg}_p^{\rm str}$ for the discriminant of the bilinear ($p$-adic height) pairing associated to the restriction to $(E(\Q)\otimes\Z_p)_0$ of the normalized quadratic form
\begin{equation}\label{eq:norm-ht}
h_{\nu/[\omega_E,\nu]_{\rm dR}}=h_\nu/[\omega_E,\nu]_{\rm dR}\nonumber
\end{equation}
for any $\nu\not\in{\rm Fil}^0D_p(E)$ (this is independent of $\nu$).
\end{defi}

%\begin{lemma}\label{lem:str-PR}
%Suppose ${\rm rank}_\Z E(\Q)\geq 1$ and $\sha(E/\Q)[p^\infty]$ is finite. Then we have the implication 
%\[
%{\rm Reg}_p^{\rm str}\neq 0\,\Longrightarrow\,{\rm Reg}_p^{\rm PR}\neq 0.
%\]
%\end{lemma}

%\begin{proof}
%This follows from \cite[Lem.~4.3]{SW}.
%\end{proof}
%is the Mordell--Weil rank of $E$. 
%and from the definitions we have the implication ${\rm Reg}_p^{\rm str}\neq 0\Longrightarrow{\rm Reg}_p^\pm\neq 0$ for both choices of sign $\pm$. 

\subsection{Perrin-Riou's formula}

The following key result is a $p$-adic analogue of the Birch--Swinnerton-Dyer conjecture for the arithmetic $p$-adic $L$-function $\mathcal{F}_p^{\rm PR}$.

\begin{thm}\label{thm:PR}
Let $E/\Q$ be an elliptic curve with good supersingular reduction at an odd prime $p$, and put $r={\rm rank}_{\Z}E(\Q)$. Then: 
\begin{itemize}
\item[(i)] $\mathcal{F}_p^{\rm PR}$ vanishes to order at least $r$ at $X=0$. 
\item[(ii)] If $\sha(E/\Q)[p^\infty]$ is finite and ${\rm Reg}_p^{\rm str}\neq 0$ then equality holds in ${\rm (i)}$, and writing 
\[
\mathcal{F}_p^{{\rm PR},(r)}:=X^{-r}\mathcal{F}_p^{\rm PR}\in D_p(E)[[X]]
\]
we have that $\mathcal{F}_p^{{\rm PR},\divideontimes}:=\mathcal{F}_p^{{\rm PR},(r)}(0)\in D_p(E)$ satisfies the equality up to a $p$-adic unit
\[
\mathcal{F}_p^{{\rm PR},\divideontimes}\;\sim_p\;(\log_p\kappa(\gamma))^{-r}\cdot(1-\varphi)^2{\rm Reg}_p^{\rm PR}\cdot\frac{\#\sha(E/\mathbb{Q})[p^\infty]\cdot{\rm Tam}(E/\mathbb{Q})}{(\#E(\mathbb{Q})_{\rm tors})^2}.
\] 
\end{itemize}
\end{thm}

\begin{proof}
This is shown in Propositions~3.4.5 and 3.4.6 in \cite{PR-points} (see also \cite[Thm.~3.1]{PR-expmath}).
\end{proof}

\begin{rem}
A result similar to Theorem~\ref{thm:PR} is obtained in \cite{PR-book} for much  more general $p$-adic representations $V$.
\end{rem}

\section{Perrin-Riou's big exponential and signed Coleman maps}

By Kobayashi's definition in \cite{kobayashi}, the local conditions at $p$ defining the signed Selmer groups ${\rm Sel}_{p^\infty}^\pm(E/\Q_n)$ are given by $E^\pm(\Q_{p,n})\otimes\bQ_p/\Z_p={\rm ker}({\rm Col}_n^\pm)^\perp$, where the superscript $\perp$ denotes the orthogonal complement under the local Tate duality
\[
\rH^1(\Q_{p,n},E[p^\infty])\times\rH^1(\Q_{p,n},T)\rightarrow\Q_p/\Z_p
\]
and ${\rm Col}_n^\pm:\rH^1(\Q_{p,n},T)\rightarrow\Z_p[{\rm Gal}(\Q_{p,n}/\Q_p)]$ are \emph{signed Coleman maps} constructed in \cite{kobayashi} using Honda's theory of formal groups. In this section, we recall a result of Lei \cite{lei-PhD} giving an independent construction of Kobayashi's 
\[
{\rm Col}^\pm:=\varprojlim_n\rH^1_{\rm Iw}(\Q_{p,\infty},T)\rightarrow\Lambda
\]
in terms of the map $\Omega_{V^*(1)}$ of Theorem~\ref{thm:PR-map}.

\subsection{Logarithm matrix}

Put 
\[
\log_p^+=\frac{1}{p}\prod_{m=1}^\infty\frac{\Phi_{2m}(1+X)}{p},\quad\quad
\log_p^-=\frac{1}{p}\prod_{m=1}^\infty\frac{\Phi_{2m-1}(1+X)}{p},
\]
where $\Phi_n(X)=\sum_{i=1}^{p-1}X^{p^{n-1}i}$ is the $p^n$-th cyclotomic polynomial.

\begin{defi}\label{def:log-matrix}
%\hfill
%\begin{itemize}
%\item[(i)] 
%\item[(ii)]
Let $\alpha,\beta\in\{\pm\sqrt{-p}\}$ be the roots of $x^2-a_px+p$ (recall that we assume $a_p=0$), and define the \emph{logarithm matrix} $M_{\rm log}\in M_{2\times 2}(\mathcal{H})$ by
\[
M_{\rm log}:=\begin{pmatrix}
{\rm log}_p^+ & {\rm log}_p^+\\
\alpha{\rm log}_p^- & \beta{\rm log}_p^-
\end{pmatrix}.
\]
%\end{itemize}
\end{defi}

\subsection{A result of Lei}

Given $\eta\in\mathbf{D}_{\rm dR}(V^*(1))$, we define the \emph{Coleman map}
\begin{equation}\label{eq:Col-eta}
{\rm Col}_\eta:\rH^1_{\rm Iw}(\Q_{p,\infty},V)\rightarrow\mathcal{H}(\Gamma) 
%,\quad\mathbf{z}\mapsto\left\langle\Omega_{V^*(1)}%(\eta\otimes(1+X)),\mathbf{z}\right\rangle_{\Q_{p,\infty}}.
\end{equation}
by $\mathbf{z}\mapsto\left\langle\Omega_{V^*(1)}(\eta\otimes(1+X)),\mathbf{z}\right\rangle_{\Q_{p,\infty}}$.  Thus, note that ${\rm Col}_\eta(\mathbf{z})=\mathscr{L}_{\mathbf{z}}(\eta)$ by definition.

%Building on the Honda theory of formal groups, in \cite{kobayashi} Kobayashi constructed \emph{signed Coleman maps}
%\[
%{\rm Col}^\pm:\rH^1_{\rm Iw}(\Q_{p,\infty},T)\rightarrow\Lambda
%\]
%sending Kato's zeta element $\mathbf{z}^{\rm Kato}$ to Pollack's $p$-adic $L$-function $L_p^\pm$. We now recall a result of Lei \cite{lei-PhD} expressing Kobayashi's ${\rm Col}^\pm$ in terms of the Coleman maps \eqref{eq:Col-eta}.

%\begin{lemma}
%There exist unique $\eta_\alpha,\eta_\beta\in\mathbf{D}_{\rm dR}(V^*(1))\simeq D_p(E)$ such that $\varphi\eta_{\alpha}=\frac{1}{\alpha}\eta_\alpha$, $\varphi\eta_\beta=\frac{1}{\beta}\eta_\beta$, and 
%$[\eta_\alpha,\omega_E]_{\rm dR}=[\eta_\beta,\omega_E]_{\rm dR}=1$.
%
%\[
%\varphi\eta_{\alpha}=\frac{1}{\alpha}\eta_\alpha,\%quad\quad\varphi\eta_\beta=\frac{1}{\beta}\eta_\beta,
%\quad\quad[\eta_\alpha,\omega_E]_{\rm dR}=[\eta_\beta,\omega_E]_{\rm dR}=1.
%\]
%\end{lemma}

%\begin{proof}
%This is shown in the proof of \cite[Thm.~16.6]{kato-295}.
%\end{proof}

\begin{thm}\label{thm:Lei-PhD}
Let $\eta_\alpha,\eta_\beta\in\mathbf{D}_{\rm dR}(V^*(1))\simeq D_p(E)$ be the unique vectors satisfying 
\[
\varphi(\eta_{\alpha})=\alpha^{-1}\eta_\alpha,\quad\quad 
\varphi(\eta_\beta)=\beta^{-1}\eta_\beta,\quad\quad 
[\eta_\alpha,\omega_E]_{\rm dR}=[\eta_\beta,\omega_E]_{\rm dR}=1.
\] 
Then for any $\mathbf{z}\in\rH^1_{\rm Iw}(\Q_{p,\infty},T)$ we have the decomposition
\begin{equation}\label{eq:dec-Col}
({\rm Col}_{\eta_\beta}(\mathbf{z}),{\rm Col}_{\eta_\alpha}(\mathbf{z}))=
({\rm Col}^-(\mathbf{z}),{\rm Col}^+(\mathbf{z}))\,M_{{\rm log}},
%\quad\textrm{where}\;M_{\rm log},
%:=
%\begin{pmatrix}
%{\rm log}_p^+(1+T) & {\rm log}_p^+(1+T)\\
%\alpha\cdot{\rm log}_p^-(1+T) & \beta\cdot{\rm log}_p^-(1+T)
%\end{pmatrix},
\end{equation}
where $M_{\rm log}\in M_{2\times 2}(\mathcal{H}(\Gamma))$ is the logarithm matrix of Definition~\ref{def:log-matrix} with $X=\gamma-1$.
\end{thm}

\begin{proof}
The existence of unique $\eta_\alpha,\eta_\beta$ satisfying the conditions in the statement is shown in \cite[Thm.~16.6]{kato-295}, while the proof of the decomposition \eqref{eq:dec-Col} is given in \cite[\S{3.4.2}]{lei-PhD}. More precisely, the Coleman maps ${\rm Col}_{\eta^\pm}$ of \eqref{eq:Col-eta} associated to the vectors $\eta^-:=\frac{\beta\eta_\alpha-\alpha\eta_\beta}{\beta-\alpha}$, $\eta^+:=\frac{\eta_\beta-\eta_\alpha}{\beta-\alpha}$ are shown to be divisible by $\log_p^\pm$, respectively\footnote{Note that for consistency with \cite{kobayashi} our signs are opposite to \cite{lei-PhD}.}, upon evaluation at any $\mathbf{z}\in\rH^1_{\rm Iw}(\Q_{p,\infty},T)$. That the maps
\[ 
{\rm Col}^\pm:\mathbf{z}\mapsto
%{\rm Col}^\pm(\mathbf{z}):=
{\rm Col}_{\eta^\pm}(\mathbf{z})/{\rm log}_p^\pm
\] 
satisfy the relation in the statement is then clear; and that they agree with %Kobayashi's 
the signed Coleman maps ${\rm Col}^\pm$ in \cite{kobayashi} follows from a relation between both constructions and the pairings $P_n$ introduced by Kurihara \cite{kurihara-sha} (see also \cite[Rem.~3.16]{lei-PhD}).
\end{proof}

\begin{rem}
Letting $f\in S_2(\Gamma_0(N))$ be the newform associated to $E$ by modularity, note that our vector $\eta_\alpha\in\mathbf{D}_{\rm dR}(V^*(1))$ corresponds to $\eta_\beta\in\mathbf{D}_{\rm dR}(V_f)\simeq\mathbf{D}_{\rm dR}(V^*)$ (with Frobenius eigenvalue $\beta$) in Lei's notation; and likewise our $\eta_\beta$ corresponds to $\eta_\alpha$ in \cite{lei-PhD}. In particular, by Kato's reciprocity law (as recalled in [\emph{op.\,cit.}, Thm.~3.10]), Kato's zeta element $\mathbf{z}^{\rm Kato}$ satisfies 
\[
{\rm Col}_{\eta_\beta}(\mathbf{z}^{\rm Kato})=L_{p,\alpha},
\] 
where $L_{p,\alpha}$ denotes the $p$-adic $L$-function of \cite{MTT} associated to $f$ and the allowable root $\alpha$; and likewise ${\rm Col}_{\eta_\alpha}(\mathbf{z}^{\rm Kato})=L_{p,\beta}$.
\end{rem}

%\section{Normalized $\varphi$-eigenvectors}

%\begin{lemma}\label{lem:calc}
%Let $\alpha,\beta\in\{\pm{\sqrt{-p}}\}$ be the two roots of $x^2-a_px+p=0$. Then
%\[
%\beta\Bigl(1-\frac{1}{\alpha}\Bigr)^2-\alpha\Bigl(1-\frac{1}{\beta}\Bigr)^2=\frac{2\alpha(p-1)}{-p},\quad\quad\Bigl(1-\frac{1}{\beta}\Bigr)^2-\Bigl(1-\frac{1}{\alpha}\Bigr)^2=\frac{4}{\alpha}.
%\]
%\end{lemma}

%\begin{proof}
%A straightforward computation.
%\end{proof}

\section{Coordinate computations}\label{sec:coordinates}

The main result of this section is the computation of the coordinates of the modified Perrin-Riou's $p$-adic regulator appearing in Theorem~\ref{thm:PR} relative to an ordered basis $(\nu_-,\nu_+)$ of $D_p(E)$ motivated by the decomposition in  Theorem~\ref{thm:Lei-PhD}.

\subsection{Dual bases} Recall that $\omega_E\in D_p(E)$ denotes the class of a fixed N\'{e}ron differential.

\begin{lemma}\label{lem:dual-basis}
Put
\[
\nu_{\alpha}:=\frac{1}{2}\,(\omega_E-\beta\varphi(\omega_E)),\quad\quad \nu_{\beta}:=\frac{1}{2}\,(\omega_E-\alpha\varphi(\omega_E)).
\]
Let $\eta_\alpha,\eta_\beta\in\mathbf{D}_{\rm dR}(V^*(1))\simeq D_p(E)$ be as in Theorem~\ref{thm:Lei-PhD}. Then $(\nu_\alpha,\nu_\beta)$ and $(\eta_\beta,\eta_\alpha)$ are dual bases of $D_p(E)$ under $[\cdot,\cdot]_{\rm dR}$, in the sense that
\begin{align*}
[\eta_\alpha,\nu_\alpha]_{\rm dR}&=
[\eta_\beta,\nu_\beta]_{\rm dR}=0,\quad\quad
[\eta_\alpha,\nu_\beta]_{\rm dR}=
[\eta_\beta,\nu_\alpha]_{\rm dR}=1.
\end{align*}
%\begin{itemize}
%\item[(i)] $e_{1/\alpha}$ and $e_{1/\beta}$ are $\varphi$-eigenvectors with eigenvalues $1/\alpha$ and $1/\beta$, respectively. 
%\item[(ii)] The normalizations
%\[
%\nu_A:=\frac{-\alpha}{\beta-\alpha}\,e_{1/\alpha},\quad\quad\nu_B:=\frac{\beta}{\beta-\alpha}\,e_{1/\beta}
%\]
%satisfy $\nu_A+\nu_B=\omega_E$. 
%\item[(iii)] $\nu_A,\nu_B$ are not multiples of $\omega_E$.
%\end{itemize}
\end{lemma}

\begin{proof}
From the relations $\varphi^2=1/p$, $\alpha+\beta=0$, and $\alpha\beta=p$ we readily see that $\varphi(\nu_\alpha)=\alpha^{-1}\nu_\alpha$ and $\varphi(\nu_\beta)=\beta^{-1}\nu_\beta$, which implies the first two equalities in the statement  
%
%and so
%\[
%[\nu_\alpha,\eta_\alpha]_{\rm dR}=[\nu_\beta,\eta_\beta]_{\rm dR}=0
%\]
by the alternating property of $[\cdot,\cdot]_{\rm dR}$. On the other hand, noting that the classes $\eta_\alpha$ and $\eta_\beta$ are necessarily multiples of $\nu_\alpha$ and $\nu_\beta$, respectively, from the defining relations $[\eta_\alpha,\omega_E]_{\rm dR}=[\eta_\beta,\omega_E]_{\rm dR}=1$ in Theorem~\ref{thm:Lei-PhD} we find
\[
\eta_\alpha=\frac{-1}{[\beta\varphi(\omega_E),\omega_E]_{\rm dR}}(\omega_E-\beta\varphi(\omega_E)),\quad\quad
\eta_\beta=\frac{-1}{[\alpha\varphi(\omega_E),\omega_E]_{\rm dR}}(\omega_E-\alpha\varphi(\omega_E)),
\]
and this yields the equalities $[\eta_\alpha,\nu_\beta]_{\rm dR}=[\eta_\beta,\nu_\alpha]_{\rm dR}=1$.
\end{proof}

\begin{lemma}\label{lem:PR-ab}
In terms of the basis $(\nu_\alpha,\nu_\beta)$ of $D_p(E)$ in  Lemma~\ref{lem:dual-basis}, we have
\[
{\rm Reg}_p^{\rm PR}=
\frac{{\rm Reg}_{\nu_\beta}}{[\omega_E,\nu_\beta]_{\rm dR}^r}\,\nu_\alpha+\frac{{\rm Reg}_{\nu_\alpha}}{[\omega_E,\nu_\alpha]_{\rm dR}^r}\,\nu_\beta.
\]
\end{lemma}

\begin{proof}
Writing ${\rm Reg}_p^{\rm PR}=a\nu_\alpha+b\nu_\beta$,  using  the defining property of ${\rm Reg}_p^{\rm PR}$, the  relation $\nu_\alpha+\nu_\beta=\omega_E$, and the fact that $[\cdot,\cdot]_{\rm dR}$ is alternating, we find
\[
\widetilde{\rm Reg}_{\nu_\alpha}=
\left[{\rm Reg}_p^{\rm PR},\nu_\alpha\right]_{\rm dR}=[b\nu_\beta,\nu_\alpha]_{\rm dR}=b[\omega_E,\nu_\alpha]_{\rm dR},
\]
and so $b=\widetilde{\rm Reg}_{\nu_\alpha}/[\omega_E,\nu_\alpha]_{\rm dR}$ 
as claimed. Similarly, we find $a=\widetilde{\rm Reg}_{\nu_\beta}/[\omega_E,\nu_\beta]_{\rm dR}={\rm Reg}_{\nu_\beta}/[\omega_E,\nu_\beta]_{\rm dR}^r$, whence the result.
\end{proof}

\subsection{The modified regulator 
$(1-\varphi)^2{\rm Reg}_p^{\rm PR}$ in coordinates}

The main result of this section is Proposition~\ref{prop:coordinates}. In the context of the analytic $p$-adic $L$-functions of \cite{sprung-JNT},  similar computations were performed by Sprung \cite{sprung-RNT}, whose notations we largely follow. 

\begin{defi}\label{def:nu-pm}
Put $Z_{\rm log}:=M_{\rm log}\vert_{X=0}=\frac{1}{p}\bigl(\begin{smallmatrix}
1 & 1\\
\alpha & \beta
\end{smallmatrix}\bigr)$, and let $N_\pm,\nu_\pm\in D_p(E)$ be the vectors given by
\[
(N_-,N_+)=(\nu_\beta,-\nu_\alpha)\begin{pmatrix}(1-\alpha^{-1})^2&\\&(1-\beta^{-1})^2\end{pmatrix}Z_{\rm log}^{-1}\cdot\det(Z_{\rm log}),\quad\quad
\begin{pmatrix}\nu_-\\\nu_+\end{pmatrix}=Z_{\rm log}\begin{pmatrix}\nu_\alpha\\\nu_\beta\end{pmatrix}.
\]
%
%\begin{itemize}
%\item[(i)] Let $\nu_\pm\in D_p(E)$ be the vectors given by
%\[
%\begin{pmatrix}\nu_-\\\nu_+\end{pmatrix}=Z_{\rm log}\begin{pmatrix}\nu_\alpha\\\nu_\beta\end{pmatrix}.
%\quad\quad
%\]
%where  $\nu_\alpha,\nu_\beta$ are as in Lemma~\ref{lem:dual-basis}.
%\item[(ii)] Let $N_\pm\in D_p(E)$ be the vectors given by
%\[
%(N_-,N_+)=(\nu_\beta,-\nu_\alpha)\begin{pmatrix}(1-\alpha^{-1})^2&\\&(1-\beta^{-1})^2\end{pmatrix}Z_{\rm log}^{-1}\cdot\det(Z_{\rm log}).
%\]
%where $\nu_\alpha,\nu_\beta$ are as in Lemma~\ref{lem:dual-basis}.
%\end{itemize}
\end{defi}
 
%Since $\det(Z_{\rm log})=\frac{\beta-\alpha}{p}\neq 0$, 
Note that the pair $(\nu_-,\nu_+)$ is a basis of $D_p(E)$. We also note that 
%
%\begin{rem}
the introduction of $N_{\pm}$ (resp. $\nu_\pm$) is motivated by the result of the computation in Proposition~\ref{prop:coordinates} (resp. the computation leading to \eqref{eq:prod-rule}) below.

\begin{lemma}\label{lem:N-pm}
We have $N_\pm\not\in{\rm Fil}^0D_p(E)$.
\end{lemma}

\begin{proof}
It suffices to show $[\omega_E,N_{\pm}]_{\rm dR}\neq 0$. Directly from the definition we have
\[
(N_-,N_+)=\bigl((1-\alpha^{-1})^2\beta\nu_\beta+(1-\beta^{-1})^2\alpha\nu_\alpha,-(1-\alpha^{-1})^2\nu_\beta-(1-\beta^{-1})^2\nu_\alpha\bigr).
\]
Thus from the relation $\omega_E=\nu_\alpha+\nu_\beta$ we obtain
\begin{equation}\label{eq:pair-N-}
\begin{aligned}
[\omega_E,N_-]_{\rm dR}&=\bigl((1-\alpha^{-1})^2\beta-(1-\beta^{-1})^2\alpha\bigr)\,[\omega_E,\nu_\beta]_{\rm dR}=\frac{2\alpha(p-1)}{-p}\cdot[\omega_E,\nu_\beta]_{\rm dR};
\end{aligned}
\end{equation}
%using Lemma~\ref{lem:calc} for the second equality; 
and similarly,
\begin{equation}\label{eq:pair-N+}
\begin{aligned}
[\omega_E,N_+]_{\rm dR}&=\bigl((1-\beta^{-1})^2-(1-\alpha^{-1})^2\bigr)\,[\omega_E,\nu_\beta]_{\rm dR}=\frac{4}{\alpha}\cdot[\omega_E,\nu_\beta]_{\rm dR}.
%=2\,[\omega_E,\nu_B].
\end{aligned}
\end{equation}
Since $[\omega_E,\nu_\beta]_{\rm dR}\neq 0$ by weak-admissibility of $D_p(E)$ (see e.g. \cite[\S{3.2}]{ghate-mezard}) and the non-degeneracy of $[\cdot,\cdot]_{\rm dR}$, the result follows.
\end{proof}

\begin{rem}
From the definitions, we directly find 
$(N_+,N_-)=(\frac{1}{p}-1)\omega_E-2\varphi(\omega_E),2\omega_E+(1-p)\varphi(\omega_E))$ 
consistently with Lemma~\ref{lem:N-pm}, but the above computations will be useful later.
\end{rem}

\begin{prop}\label{prop:coordinates}
Suppose $r={\rm rank}_\Z E(\Q)\geq 1$. Then the coordinates $(c_+,c_-)$  of $(1-\varphi)^2{\rm Reg}^{\rm PR}_p\in D_p(E)$ with respect to the ordered basis $(\nu_+,\nu_-)$ are given by
\[
(c_+,c_-)=\Bigl(2\,\frac{{\rm Reg}_{N_+}}{[\omega_E,N_+]_{\rm dR}^r},(p-1)\,\frac{{\rm Reg}_{N_-}}{[\omega_E,N_-]_{\rm dR}^r}\Bigr).
\] 
\end{prop}

\begin{proof}
We begin by noting that the association $\nu\mapsto\widetilde{\rm Reg}_\nu={\rm Reg}_\nu/[\omega_E,\nu]^{r-1}_{\rm dR}$ is linear in $\nu\in D_p(E)\smallsetminus{\rm Fil}^0D_p(E)$ (whenever defined), and by Lemma~\ref{lem:N-pm} and its proof the quantities $\widetilde{\rm Reg}_{\nu_\alpha}$, $\widetilde{\rm Reg}_{\nu_\beta}$, $\widetilde{\rm Reg}_{N_+}$, $\widetilde{\rm Reg}_{N_-}$ are all defined. Thus from the expression for ${\rm Reg}^{\rm PR}_p$ in Lemma~\ref{lem:PR-ab} we obtain
\begin{align*}
(1-\varphi)^2{\rm Reg}_p^{\rm PR}&=\Bigl(\frac{{\rm Reg}_{\nu_\beta}}{[\omega_E,\nu_\beta]_{\rm dR}^r},\frac{{\rm Reg}_{\nu_\alpha}}{[\omega_E,\nu_\alpha]_{\rm dR}^r}\Bigr)
\begin{pmatrix}(1-\alpha^{-1})^2&\\&(1-\beta^{-1})^2\end{pmatrix}
\begin{pmatrix}\nu_\alpha\\\nu_\beta\end{pmatrix}\\
&=\Bigl(\frac{\widetilde{\rm Reg}_{\nu_\beta}}{[\omega_E,\nu_\beta]_{\rm dR}},\frac{\widetilde{\rm Reg}_{\nu_\alpha}}{[\omega_E,\nu_\alpha]_{\rm dR}}\Bigr)\begin{pmatrix}(1-\alpha^{-1})^2&\\&(1-\beta^{-1})^2\end{pmatrix}Z_{\rm log}^{-1}\begin{pmatrix}\nu_-\\\nu_+\end{pmatrix}\\
&=\Bigl(\frac{\widetilde{\rm Reg}_{N_-}}{[\omega_E,\nu_\beta]_{\rm dR}},\frac{\widetilde{\rm Reg}_{N_+}}{[\omega_E,\nu_\beta]_{\rm dR}}\Bigr)\begin{pmatrix}\nu_-\\\nu_+\end{pmatrix}\frac{p}{\beta-\alpha},
\end{align*}
using the relations $\widetilde{\rm Reg}_{-\nu_\alpha}=-\widetilde{\rm Reg}_{\nu_\alpha}$ and $[\omega_E,\nu_\beta]_{\rm dR}=-[\omega_E,\nu_\alpha]_{\rm dR}$ %(see Lemma~\ref{lem:phi-eigen}) 
for the last equality. 
%Thus it remains to show the equalities
%\[
%\frac{[\omega_E,N_-]}{[\omega_E,\nu_B]}=p-1,\quad\quad\frac{[\omega_E,N_+]}{[\omega_E,\nu_B]}=2
%\]
%which follow immediately from the definitions and Lemma~\ref{lem:calc}.
In light of \eqref{eq:pair-N-} and \eqref{eq:pair-N+}, this yields the result.
\end{proof}

\section{Proof of the main result}

%\begin{prop}\label{cor:Lei}
%Let $\mathbf{z}\in\mathbb{H}^1(T)$ be any nonzero element. 
%In matrix form, the $D_p(E)$-valued measure $\mathcal{F}_p^{\rm PR}$ can be rewritten as
%\[
%\mathcal{F}_p^{\rm PR}=(\mathcal{F}_p^-,\mathcal{F}_p^+)M_{\rm log}\begin{pmatrix}\nu_A\\\nu_B\end{pmatrix}.
%\]
%\end{prop}

%\begin{proof}
%After Theorem~\ref{thm:Lei-PhD}, it suffices to check that $(\eta_\beta,\eta_\alpha)$ is the basis of $D_p(E)$ dual to $(\nu_A,\nu_B)$ under the duality $[\cdot,\cdot]_{\rm dR}$, and this follows by a straightforward computation.
%\end{proof}

As in the Introduction,  we denote by ${\rm Reg}_p^\pm\in\Q_p$ the $p$-adic regulator of Definition~\ref{def:reg-nu} associated to $h_{N_\pm/[\omega_E,N_\pm]_{\rm dR}}$, so  
\[
{\rm Reg}_p^\pm:={\rm Reg}_{N_\pm/[\omega_E,N_\pm]_{\rm dR}}=\frac{{\rm Reg}_{N_\pm}}{[\omega_E,N_\pm]^r_{\rm dR}},
\]
where $r={\rm rank}_\Z E(\Q)$.

\subsection{Signed arithmetic $p$-adic $L$-functions}

For $\mathbf{z}\in\mathbb{H}^1(T)$ any non-torsion element, we put
\begin{equation}\label{eq:F-pm}
\mathcal{F}_p^{\pm}:={\rm Col}^\pm(\mathbf{z}_p)\cdot\frac{g_{\rm str}}{h_{\mathbf{z}}}\in\Z_p[[X]],
\end{equation}
where 
%$\mathbf{z}_p$ denotes the image of $\mathbf{z}$ under the restriction map $\mathbb{H}^1(T)\rightarrow\mathbb{H}^1_{\rm Iw}(\Q_{p,\infty},T)$, and 
$g_{\rm str}$ and $h_{\mathbf{z}}$ are as in Definition~\ref{def:arith-Lp}.
%
%
%\begin{prop}\label{prop:PT}
%The power series $\mathcal{F}_p^\pm(E,T)$ gives a generator of ${\rm char}_\Lambda(X^\pm(E/\Q_\infty))$. 
%\end{prop}
%
%\begin{proof}
%As explained in \cite[\S{6.4}]{lei-PhD}, Poitou--Tate duality gives rise to the four-term exact sequence
%\[
%0\rightarrow\mathbb{H}^1_S(T_pE)\rightarrow{\rm Im}({\rm Col}^\pm)\rightarrow X^\pm(E/\Q_\infty)\rightarrow{\rm Sel}_{p^\infty}^0(E/\Q_\infty)^\vee\rightarrow 0.
%\]
%%By Kato's explicit reciprocity law (due to Kobayashi), 
%This induces
%\begin{equation}\label{eq:PT-z}
%0\rightarrow\frac{\mathbb{H}^1_S(T_pE)}{(\mathbf{z})}\rightarrow\frac{{\rm Im}({\rm Col}^\pm)}{({\rm Col}^\pm({\rm res}_p(\mathbf{z})))}\rightarrow X^\pm(E/\Q_\infty)\rightarrow{\rm Sel}_{p^\infty}^0(E/\Q_\infty)^\vee\rightarrow 0.
%\end{equation}
%Since the maps ${\rm Col}^\pm$ have pseudo-null cokernel (see \cite[Thm.~6.2]{kobayashi}), taking characteristic ideals in \eqref{eq:PT-z}, the result follows by multiplicativity.
%\end{proof}
%
The following is the main result of this note.

\begin{thm}\label{thm:main}
Let $E/\Q$ be an elliptic curve with good supersingular reduction at an odd prime $p$ with $a_p=0$. Put 
\[
r:={\rm rank}_\Z E(\Q),
\]
and suppose $r\geq 1$. Then:
\begin{itemize}
\item[(i)] $\varrho:=\min\{{\rm ord}_X(\mathcal{F}_p^+),{\rm ord}_X(\mathcal{F}_p^-)\}\geq r$.
\item[(ii)] If $\sha(E/\Q)[p^\infty]$ is finite and ${\rm Reg}_p^{\rm str}\neq 0$, then equality holds in ${\rm (i)}$ and the leading coefficient $(\mathcal{F}_p^{+,\divideontimes},\mathcal{F}_p^{-,\divideontimes})$ of the vector $(\mathcal{F}_p^+,\mathcal{F}_p^-)\in\Z_p[[X]]^{\oplus{2}}$ is given up to a $p$-adic unit by  
\[
(\mathcal{F}_p^{+,\divideontimes},\mathcal{F}_p^{-,\divideontimes})\;\sim_p\;(\log_p\kappa(\gamma))^{-r}\cdot({\rm Reg}_p^+,{\rm Reg}_p^-)\cdot\frac{\#\sha(E/\mathbb{Q})[p^\infty]\cdot{\rm Tam}(E/\mathbb{Q})}{(\#E(\mathbb{Q})_{\rm tors})^2}.
\]
\end{itemize}
\end{thm}

\begin{proof}
We begin by noting that by Theorem~\ref{thm:Lei-PhD} and Lemma~\ref{lem:dual-basis}, we can rewrite the arithmetic $p$-adic $L$-function $\mathcal{F}_p^{\rm PR}\in D_p(E)[[X]]$  of Definition~\ref{def:arith-Lp} in matrix form as
\begin{equation}\label{eq:cor-Lei}
\mathcal{F}_p^{\rm PR}=(\mathcal{F}_p^-,\mathcal{F}_p^+)M_{\rm log}\begin{pmatrix}\nu_\alpha\\\nu_\beta\end{pmatrix}.
%\quad\textrm{where}\quad\mathcal{F}_p^{\pm}:={\rm Col}%^\pm(\mathbf{z}_p)\cdot\frac{g_{\rm str}}{h_{\mathbf{z}}}%\in\Z_p[[X]]
\end{equation}
%
%Indeed, this is immediate from the fact that $(\eta_\beta,\eta_\alpha)$ is the basis of $D_p(E)$ dual to $(\nu_A,\nu_B)$ under the duality $[\cdot,\cdot]_{\rm dR}$, which follows by a straightforward computation.
In particular, since ${\rm log}_p^\pm\vert_{X=0}\neq 0$, 
from \eqref{eq:cor-Lei} and the product rule we readily find
\begin{equation}\label{eq:der-t}
\frac{d^t}{dX^t}\mathcal{F}_p^{\rm PR}\Bigr\vert_{X=0}=\Bigl(\frac{d^t}{dX^t}\mathcal{F}_p^{-}\Bigr\vert_{X=0},\frac{d^t}{dX^t}\mathcal{F}_p^{+}\Bigr\vert_{X=0}\Bigr)Z_{\rm log}\begin{pmatrix}\nu_\alpha\\\nu_\beta\end{pmatrix}
\end{equation}
for all $t\geq 0$.
Since the matrix $Z_{\rm log}$ is invertible, this shows that 
\begin{equation}\label{eq:PR-varrho}
{\rm ord}_X(\mathcal{F}_p^{\rm PR})=\varrho,
\end{equation} 
and therefore the proof of part (i) follows from Theorem~\ref{thm:PR}(i). 
 For the proof of part (ii), suppose $\sha(E/\Q)[p^\infty]$ is finite and ${\rm Reg}_p^{\rm str}\neq 0$; then $\varrho=r$ by \eqref{eq:PR-varrho} and Theorem~\ref{thm:PR}(ii). 
%and also ${\rm Reg}_p^{\rm PR}\neq 0$ by Lemma~\ref{lem:str-PR}.  
%
%Thus for the proof of part (i) it remains to see that 
%\begin{equation}\label{eq:+=-}
%{\rm ord}_{T=0}\mathcal{F}_p^+(E,T)={\rm ord}_{T=0}\mathcal{F}_p^-(E,T).
%\end{equation} 
%% (in this case, note that \emph{a priori} we could have ${\rm ord}_{T=0}\mathcal{F}_p^-(E,T)\neq{\rm ord}_{T=0}\mathcal{F}_p^+(E,T)$, although we shall show these two are in fact equal). 
%Suppose the condition for equality in (i) holds, 
Now put 
%$\mathcal{F}_p^{{\rm PR},(\varrho)}\in D_p(E)[[X]]$ and  $\mathcal{F}_p^{\pm,(\varrho)}\in\Z_p[[X]]$ by the equalities
\[
\mathcal{F}_p^{{\rm PR},(r)}:=X^{-r}\mathcal{F}_p^{{\rm PR}}\in D_p(E)[[X]],\quad\quad
\mathcal{F}_p^{\pm,(r)}:=X^{-r}\mathcal{F}_p^{\pm}\in\Z_p[[X]],
\]
%\[
%\mathcal{F}_p^{\rm PR}=\mathcal{F}_p^{{\rm PR},(\varrho)}\cdot X^\varrho,\quad\quad
%\mathcal{F}_p^{\pm}=\mathcal{F}_p^{\pm,(\varrho)}\cdot X^{\varrho}.
%\] 
%in particular, $\mathcal{F}_p^{{\rm PR},(\varrho)}(0)=\mathcal{F}_p^{{\rm PR},\divideontimes}$.
and note that \eqref{eq:der-t} yields the middle equality in the chain
\begin{equation}\label{eq:prod-rule}
\mathcal{F}_p^{{\rm PR},\divideontimes}=\mathcal{F}_p^{{\rm PR},(r)}(0)=
\bigl(\mathcal{F}_p^{-,(r)}(0),\mathcal{F}_p^{+,(r)}(0)\bigr)\begin{pmatrix}\nu_-\\\nu_+\end{pmatrix}
=\bigl(\mathcal{F}_p^{-,\divideontimes},\mathcal{F}_p^{+,\divideontimes}\bigr)\begin{pmatrix}\nu_-\\\nu_+\end{pmatrix}.
\end{equation}
On the other hand, by Theorem~\ref{thm:PR}(ii) and Proposition~\ref{prop:coordinates} we have that the coordinates $(d_+,d_-)$ of $\mathcal{F}_p^{{\rm PR},\divideontimes}$ with respect to $(\nu_+,\nu_-)$ are given up to a $p$-adic unit by
\begin{equation}\label{eq:g}
(d_+,d_-)\;\sim_p\;(\log_p\kappa(\gamma))^{-r}\cdot({\rm Reg}_p^+,{\rm Reg}_p^-)\cdot\frac{\#\sha(E/\mathbb{Q})[p^\infty]\cdot{\rm Tam}(E/\mathbb{Q})}{(\#E(\mathbb{Q})_{\rm tors})^2},\nonumber
\end{equation}
which together with \eqref{eq:prod-rule} concludes the proof of part (ii).
%
%Since the nonvanishing of ${\rm Reg}_p^{\rm str}$ implies that of ${\rm Reg}_p^{\rm PR}$ by Lemma~\ref{lem:str-PR}, from \eqref{eq:prod-rule} and \eqref{eq:g} we see that $(\mathcal{F}_p^{+,(\varrho)}(0),\mathcal{F}_p^{-,(\varrho)}(0))\neq(0,0)$, and so
%\[
%{\rm ord}_X\mathcal{F}_p^+={\rm ord}_X\mathcal{F}_p^-={\rm ord}_X\mathcal{F}_p^{\rm PR}=\varrho.
%\]
%In particular, it follows that \eqref{eq:prod-rule} can be rewritten as
%$(d_+,d_-)=(\mathcal{F}_p^{+,\divideontimes},\mathcal{F}_p^{-,\divideontimes})$, which together with \eqref{eq:g} yields the proof of part (ii).
\end{proof}

\subsection{Proof of Theorem~\ref{thm:main-intro}}

%As a consequence, we can confirm the algebraic $p$-adic Birch--Swinnerton-Dyer conjecture as stated in \cite[Conjecture~3.15]{KR-I}.

\begin{proof}[Proof of Theorem~\ref{thm:main-intro}]
In view of Theorem~\ref{thm:main}, it suffices to show that the power series $\mathcal{F}_p^\pm\in\Z_p[[X]]$ introduced in \eqref{eq:F-pm} %Definition~\ref{def:F-pm} 
generates the characteristic ideal of $\mathcal{X}^\pm(E/\Q_\infty)$. 

As explained in \cite[\S{6.4}]{lei-PhD}, Poitou--Tate duality gives rise to the four-term exact sequence
\[
0\rightarrow\mathbb{H}^1(T)\rightarrow{\rm Im}({\rm Col}^\pm)\rightarrow\mathcal{X}^\pm(E/\Q_\infty)\rightarrow{\rm Sel}_{p^\infty}^{\rm str}(E/\Q_\infty)^\vee\rightarrow 0.
\]
%By Kato's explicit reciprocity law (due to Kobayashi), 
This induces
\begin{equation}\label{eq:PT-z}
0\rightarrow\frac{\mathbb{H}^1(T)}{(\mathbf{z})}\rightarrow\frac{{\rm Im}({\rm Col}^\pm)}{({\rm Col}^\pm(\mathbf{z}_p))}\rightarrow\mathcal{X}^\pm(E/\Q_\infty)\rightarrow{\rm Sel}_{p^\infty}^{\rm str}(E/\Q_\infty)^\vee\rightarrow 0.
\end{equation}
Since the $\Lambda$-linear maps ${\rm Col}^\pm$ have pseudo-null cokernel by \cite[Thm.~6.2]{kobayashi}, we see that the second term in \eqref{eq:PT-z} has characteristic ideal generated by ${\rm Col}^\pm(\mathbf{z}_p)$, and so the fact that $\mathcal{F}_p^\pm$ has the desired property follows by multiplicativity.
\end{proof}

\bibliography{pBSD-references}
\bibliographystyle{alpha}

\end{document}